\documentclass[a4paper, 11pt]{article}
\usepackage[utf8x]{inputenc}
\usepackage{graphicx, amsfonts,amsmath,amssymb,amsthm}
\usepackage{natbib}
\usepackage{hyperref,cleveref}
\usepackage[a4paper,top=3cm,bottom=2cm,left=2.5cm,right=2.5cm,marginparwidth=5cm]{geometry}
\usepackage{enumitem}
\usepackage{commands}
\usepackage{tikzSettings}
\usepackage{caption}
\usepackage{subcaption}
\providecommand{\keywords}[1]{\textbf{\textit{Keywords}} #1}

\author{Victor Cohen, Axel Parmentier \\
Ecole des Ponts Paristech, \textsc{Cermics}, Universit\'e Paris-Est, Marne-la-Vall\'ee, France}
\title{Two generalizations of Markov blankets}

\date{
		\today}

\begin{document}
\maketitle

\begin{abstract}
\noindent In a probabilistic graphical model on a set of variables $V$, the Markov blanket of a random vector $B$ is the minimal set of variables conditioned to which $B$ is independent from the remaining of the variables $V \backslash B$.
We generalize Markov blankets to study how a set $C$ of variables of interest depends on~$B$.
Doing that, we must choose if we authorize vertices of $C$ or vertices of $V \backslash C$ in the blanket. 
We therefore introduce two generalizations.
The Markov blanket of $B$ in $C$  is the minimal subset of $C$ conditionally to which $B$ and $C$ are independent.
It is naturally interpreted as the inner boundary through which $C$ depends on $B$, and finds applications in feature selection.
The Markov blanket of $B$ in the direction of $C$ is the nearest set to $B$ among the minimal sets conditionally to which ones $B$ and $C$ are independent, and finds applications in causality.
It is the outer boundary of $B$ in the direction of $C$.
We provide algorithms to compute them that are not slower than the usual algorithms for finding a d-separator in a directed graphical model. 
All our definitions and algorithms are provided for directed and undirected graphical models.
\end{abstract}





\section{Introduction}
\label{sec:introduction}

A distribution on a set of variables $V$ factorizes as a probabilistic graphical model on a graph $G = (V,A)$ if variables in $V$ satisfy some independences that are encoded by $G$.
Given a set $B$ of variables in $V$, the Markov blanket of $B$ is the \emph{boundary in $V \backslash B$ through which $B$ and $V\backslash B$ are dependent}. More formally, it is the smallest subset $M$ of $V\backslash B$ such that 
\begin{equation}\label{eq:MB_general_characterization}
	B \indep V\backslash \left( B \cup M \right) | M
\end{equation}
for any distribution that factorizes as a probabilistic graphical model on $G$, 
where, given three random vectors $X$, $Y$, and $Z$, we denote by  $$X \indep Y | Z$$
 the fact that $X$ is independent from $Y$ given $Z$.
As illustrated on Figure~\ref{fig:usualMB}, $\mb(B)$ corresponds to the ``outer boundary'' of $B$, and $\mb(V\backslash B)$ to its ``inner boundary''. 
The Markov Blanket of $B$ is the smallest set of variables of $V \backslash B$ containing all the information about $B$ that is in~$V\backslash B$ \citep{Pellet:2008}. 

In this paper, \emph{we introduce two generalizations of Markov blankets to model how a subset of variables depends on another.}
The first is the \emph{Markov blanket of $B$ in $C$}, which we denote by $\mb_C(B)$. 
It is the smallest subset $M$ of $C$ such that $ B \indep C \backslash M \big| M$.
The second is the \emph{Markov blanket of $B$ in the direction of $D$}, which we denote by $\mb(B \rightarrow D)$.
Among the sets $M$ in $V \backslash B$ such that $B \indep D | M$ and that are minimal for inclusion, it is the ``nearest'' to $B$. 
Figure~\ref{fig:innerAndOuterBoundaries} illustrated how these notions can be interpreted as inner and outer boundaries.

We introduce $\mb_C(B)$ and $\mb(B\rightarrow D)$ in directed and undirected graphical models. 
We characterize $\mb_C(B)$ and $\mb(B\rightarrow D)$ in terms of separation and d-separation, which provides polynomial time algorithms to compute them.
Our characterizations can take into account the fact that some variables $E$ have been observed. 

\begin{ex} \emph{Feature selection and Markov blanket of $B$ in $C$.}
Suppose that we observe the variables in $C$ and want to predict the value of the variables in $B$ \citep{Kohavi:1997}.
Feature selection aims at finding in $C$ the most relevant variables to make the prediction on $B$.
If we know that $B$ and $C$ are composed of vertices of a larger probabilistic graphical model $G$, then the Markov blanket of $B$ in $C$ is the set of variables we are interested in: it is the smallest subset of $C$ that contains all the information on $B$ that is in $C$.

If we cannot observe the variables in $C$ but we can observe all the other variables in $V \backslash C$, 
we need to find a minimal set in $V \backslash (C \cup B)$ that contains all ``effect'' of $C$ on $B$: the Markov blanket of $B$ in the direction of $C$.
\end{ex}

\begin{ex}\label{ex:causality} \emph{Causality and Markov blanket of $B$ in the direction of $D$.}
Suppose that a medical doctor observes that one patient that suffers from disease $D$ has an abnormally blood sugar level~$B$.
The fact that $B$ and $D$ are correlated does not mean that $B$ has an influence on~$D$.
Indeed, if $D$ might cause $B$, it might also be that $B$ and $D$ are both caused by another factor.
Fixing $B$ will cure the patient from $D$ only if $B$ is a cause of $D$.
Counting the number of patients suffering from $D$ among those having $B$ indicates the correlation of $B$ and $D$, i.e.,~the conditional probability $\bbP(D|B)$ of $D$ given $B$, but not the causal effect of $B$ and $D$.
To measure this \emph{causal effect}, we need to compute the conditional probability of $D$ given $B$ in an experiment where, all other things being equal, parameter $B$ is controlled.
We denote it by $\bbP(D | \mathrm{do}(B))$.
If $B$ and $D$ are random variables of a probabilistic graphical model, causality theory enables to identify if the causal effect $\bbP(D | \mathrm{do}(B))$ can be computed from historical data without setting up a new experiment, and to compute it when it is possible. 
\citet{PearlShpitserP12} introduce an algorithm which returns all the causal effects $\bbP(D | \mathrm{do}(B))$ that can be computed in a directed graphical model. 
This algorithm, 
which uses the back-door criterion \citep{Pearl93graphicalmodels}, 
requires to compute a d-separator between $(\dsc{B} \cap \asc{D}) \cup D$ and $B$ in the graph where we remove arcs outgoing from $B$, where $\asc{M}$ and $\dsc{M}$ respectively denote the ascendants and descendants of a set of a vertices $M$.
Let $S$ be such d-separator. Computing the causal effect of $B$ on $D$ becomes equivalent to computing conditional probabilities and marginals in a directed graphical model (\citet[e.g. Theorem 1.14]{Lauritzen99causalinference}) :
$$\bbP(D | \mathrm{do}(B=b)) = \sum_{s} \bbP(D | S=s,B=b)\bbP(S=s)$$
Hence, we need to perform an inference task to compute the probabilities in the sum above.
This latter inference problem is easier if the d-separator is small and near to $B$.
The Markov Blanket of $(\dsc{B} \cap \asc{D}) \cup D$ in the direction of $B$ is therefore an excellent candidate as d-separator~$S$: it is the nearest from $(\dsc{B} \cap \asc{D}) \cup D$ among all the minimal d-separator between $(\dsc{B} \cap \asc{D}) \cup D$ and~$B$.
\end{ex}



\begin{figure}
\begin{center}



\begin{tikzpicture}
\def\w{0.35}
\def\h{0.35}

\draw[thick, fill=gray!20] (0,0) rectangle (11*\w,8*\h);
\draw[thick] (4*\w,2*\h) rectangle (9*\w,6*\h);
\node at (4.7*\w,2.7*\h) {$B$};

\node at (0.6*\w,0.6*\h) {$V$};

\end{tikzpicture}
\quad
\begin{tikzpicture}
\def\w{0.35}
\def\h{0.35}

\draw[thick, fill=gray!20] (0,0) rectangle (11*\w,8*\h);
\draw[dashed,fill=cyan!75!magenta,opacity=0.5] (3*\w,1*\h) rectangle (10*\w,7*\h);
\fill[fill=blue!20] (4*\w,2*\h) rectangle (9*\w,6*\h);
\draw[thick,fill=cyan, fill opacity = 0.5] (4*\w,2*\h) rectangle (9*\w,6*\h);
\draw[dashed,fill=blue!20] (5*\w,3*\h) rectangle (8*\w,5*\h);
\node at (6*\w,6.5*\h) {$\mb(B)$};
\node at (6.5*\w,2.6*\h) {$\mb(V\backslash B)$};


\end{tikzpicture}

\end{center}
\caption{Markov blanket of $B$ and $V \backslash B$}
\label{fig:usualMB}
\end{figure}
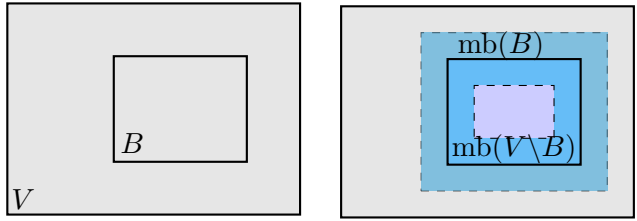

\begin{figure}
\begin{center}
\begin{tikzpicture}

\def\h{1}

\fill[fill=gray!20] (0,-0.05) rectangle (11+2,2*\h+0.05);
\draw[fill=blue!20, thick] (1,0) rectangle (4,2*\h);
\draw[dashed,fill=cyan,opacity=0.5] (3,0) rectangle (4,2*\h);
\draw[dashed,fill=cyan!75!magenta,opacity=0.5] (4,0) rectangle (5,2*\h);

\draw[fill=red!20, thick] (7+2,0) rectangle (10+2,2*\h);
\draw[dashed,fill=cyan!25!magenta,opacity=0.5] (6+2,0) rectangle (7+2,2*\h);
\draw[dashed,fill=magenta,opacity=0.5] (7+2,0) rectangle (8+2,2*\h);

\node at (0.2,0.2) {$V$};
\node at (1.3,1*\h) {$B$};
\node at (3.2,1*\h) {$\mb_B(C)$};
\node at (5.1,1*\h) {$\mb(B\rightarrow C)$};

\node at (7.9,1*\h) {$\mb(C\rightarrow B)$};
\node at (9.9,1*\h) {$\mb_C(B)$};
\node[align=right] at (11.7,1*\h) {$C$};


\end{tikzpicture}
\end{center}
\caption{Markov blankets as boundaries between $B$ and $C$: $\mb_C(B)$ is the inner boundary of $C$ in the direction of $B$, $\mb(C \rightarrow B)$ is the outer boundary of $C$ in the direction of $B$, $\mb_B(C)$ is the inner boundary of $B$ in the direction of $C$ and $\mb(B \rightarrow C)$ is the outer boundary of $B$ in the direction of $C$}
\label{fig:innerAndOuterBoundaries}
\end{figure}
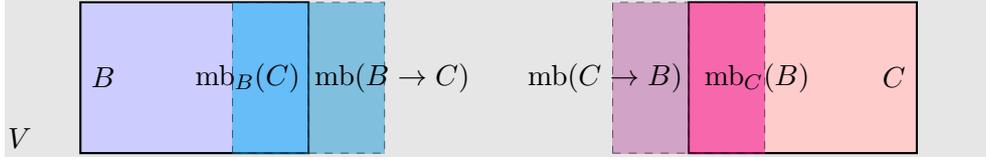


Section~\ref{sec:preliminaries_on_probabilistic_graphical_models} introduces the notions and notations we need on directed and undirected graphical models, as well as a literature review on Markov blankets. 
Section~\ref{sec:markov_blanket_in_a_set} the introduces the Markov blanket of $B$ in $C$, and Section~\ref{sec:directional_markov_blanket} the Markov blanket of $B$ in the direction of $D$.

\section{Preliminaries on probabilistic graphical models}
\label{sec:preliminaries_on_probabilistic_graphical_models}
\subsection{Graphs}
\label{sub:graphs}

A \emph{graph} is a pair $G = (V,A)$ where $V$ is a finite set and $A$ is a family of unordered pairs from $V$. 
A \emph{vertex} $v$ is an element of $V$.
In an \emph{undirected graph}, the pairs $e=(u,v)$  in $A$  are unordered and called \emph{edges}.
In a \emph{directed graph}, the pairs $a = (u,v)$ in $A$ are ordered and called \emph{arcs}.

A $u$-$v$ \emph{path} $P$ in a graph is a sequence of vertices $v_0,\ldots,v_k$ such that $v_0=u$, $v_k = v$, and $(v_{i-1},v_i)$ belongs to $A$ for each $i$ in $[k]$.
Remark that if $v_1,\ldots,v_k$ is a path in an undirected graph, then $v_k,\ldots,v_1$ is also a path.
But if $v_1,\ldots,v_k$ is a path in an directed graph, then $v_k,\ldots,v_1$ is generally not a path.
A \emph{cycle} in a graph is a path $v_0,\ldots,v_k$ such that $k>0$ and $v_0=v_k$.
An directed graph is \emph{acyclic} if it has no cycle.
A $u$-$v$ \emph{trail} in an acyclic directed graph is a sequence of vertices $v_1,\ldots,v_k$ such that $v_0=u$, $v_k = v$, and either $(v_{i-1},v_i)$ or $(v_i,v_{i-1})$ belongs to $A$ for each $i$ in $[i]$.
A vertex $v_i$ in a trail $v_0,\ldots,v_k$ is a \emph{v-structure} if $0<i<k$ and $(v_{i-1},v_i)$ and $(v_{i+1},v_i)$ belong to $A$.
A \emph{clique} in an undirected graph is a subset $C$ of vertices of $V$ such that, if $u$ and $v$ are two distinct elements of $V$, then $(u,v)$ belongs to $A$.

Let $G$ be an acyclic directed graph. A \emph{parent} 
 of a vertex $v$ is a vertex $u$ such that $(u,v)$ 
 belongs to~$A$; we denote by $\prt{v}$ the set of parents of $v$.
A vertex $u$ is an \emph{ascendant} (resp.~a \emph{descendant}) of $v$ if there exists a $u$-$v$~path (resp.~a $v$-$u$ path). 
We denote respectively $\asc{v}$ and $\dsc{v}$ the set of ascendants and descendants of $v$.
Finally, let $\casc{v} = \{v\}\cup \asc{v}$, and $\cdsc{v} = \{v\} \cup \dsc{v}$.
For a set of vertices $C$, the parent set of $C$, again denoted by $\prt{C}$, is the set of vertices $u$ that are parents of a vertex $v\in C$.
We define similarly $\asc{C}$, and $\dsc{C}$.

We associate with each vertex $v$ in $V$ a random variable $X_v$ taking its value in a finite set $\calX_v$.
For any subset $A$ of $V$, we define $X_A$ as the subvector $(X_v)_{v \in A}$, and $\calX_A$ as the Cartesian product $\bigotimes_{v\in A}\calX_v$.

\subsection{Undirected graphical model}
\label{sub:undirected_graphical_model}

Given an undirected graph $G=(V,A)$, a probability distribution $\bbP$ on $\calX_V$ factorizes as an \emph{undirected graphical model} on $G$ if there exists a collection $\calC$ of cliques of $G$, and mappings $\psi_C : \calX_C \rightarrow \bbR^+$ for each $C$ in $\calC$ such that 
$$ \bbP(X_V = x_V) = \frac{1}{Z} \prod_{C \in \calC} \psi_C(x_C), $$
where $Z$ is a constant ensuring that $\bbP$ is a probability distribution. 
Vertices of a graphical model corresponds to random variables, and sets of vertices to random vectors.

A $u$-$v$ path $P$ is \emph{active} given a subset of vertices $M$ if no vertex of $P$ is in $M$.
A set of vertices $M$ separates two sets of vertices $X$ and $Y$ if there is no active path between a vertex of $X$ and a vertex of $Y$, which we denote by
$$ X \perp Y | M. $$
Given three random vectors $X$, $Y$, and $M$, graphical model theory tells us that $X$ is independent from $Y$ given $M$ for any distribution that factorizes as a graphical model on $G$ if and only if $M$ separates $X$ and $Y$ (see e.g.~Theorem 4.3 of \citet{koller2009probabilistic}).




\medskip

We are interested in independences of probabilistic graphical models $G$, that is, independences that are true for any distribution that factorizes as a graphical models.
Such independences must therefore be characterized only in terms of the structure of $G$, that is, in terms of separation and d-separating.

\subsection{Directed graphical models}
\label{sub:directed_graphical_models}

Let $G = (V,A)$ be an acyclic directed graph.
A conditional distribution of $v$ given its parent is a mapping $p_{v|\prt{v}} : \calX_v \times \calX_{\prt{v}} \rightarrow \bbR_+$ such that, for each $x_{\prt{v}}$ in $\calX_{\prt{v}}$, the mapping $x_v \mapsto p_{v|\prt{v}}(x_{v},x_{\prt{v}})$ is a probability distribution.
A distribution $\bbP$ on $\calX_v$ factorizes as a \emph{directed graphical model} on $G$ if there exists conditional distributions $p_{v|\prt{v}}$ such that
$$ \bbP(x_V) = \prod_{v \in V} p_{v|\prt{v}}(x_v,x_{\prt{v}}).$$
Given a subset $M$ in $V$, a $u$-$v$ trail $P$ is \emph{active} if and only if any vertex $v$ in $P$ that is not a v-structure does not belong to $P$, and any vertex $v$ in $P$ that is a v-structure is such that $\cdsc{v} \cap M \neq \emptyset$.
Given three random vectors $X$, $Y$, and $M$, then $M$ \emph{d-separates} $X$ and $Y$ if there is no active trail between $X$ and $Y$ that is active given $M$, which we again denote by
$$ X \perp Y | M. $$ 
Three random vectors $X$, $Y$, and $M$ are such that $X$ is independent from $Y$ given $M$ for any distribution that factorizes as a graphical model on $G$ is and only if $X$ is d-separated from $Y$ given $M$ (see e.g.~Theorems 3.4 and 3.5 of \citet{koller2009probabilistic}).

\subsection{Markov blankets and separators}
\label{sub:markov_blankets_and_separators}

A \emph{separator} (resp.~a d-separator) between two set of vertices $B$ and $D$ given an evidence set $E$ in an undirected (resp.~directed) graphical model $G$ is a set of vertices $M$ that separates (resp.~d-separates) $B$ and $D$.
A (d-)separator $M$ between two sets of vertices $B$ and $D$ given an evidence set $E$ is \emph{minimal} if for any strict subset $M'$ of $M$, $M' \cup E$ does not (d-)separate $C$ and $D$. 

The \emph{Markov blanket }$\mb(B)$ of $B$ is the smallest (d-)separator $M \subseteq V \backslash B$ of $B$ and $V \backslash B$. By smallest, we mean that any (d-)separator $M \subseteq V \backslash B$ of $B$ and $V \backslash B$ contains $\mb(B)$.

\subsection{Literature review}
\label{sub:literature_review}

Markov blankets are built on the fact that independences in a graphical model are characterized in terms of separation and d-separation. \citet{Lauritzen1990} introduces the notion \emph{separation} in a undirected graphical model, which coincides with the separation in graph theory. The author also introduces the notion of \emph{d-separation} in a directed graphical model. \citet{Geiger1990dseparation} presents the Bayes-ball algorithm that checks if two vertices in a directed graph $G=(V,A)$ are \emph{d-separated} by a given set of vertices in $O(|V|+|A|)$.
\citet{Pearl1988} introduced the notion of Markov Blanket in the context of causal structure learning, under the name Markov boundary.
Given samples a set of random variables, causal structure learning aims at learning a directed graphical model that represents the causal links between the random variables. 
\citet{Pearl1988} and \citet{Spirtes2000} characterize graphically the Markov blanket: in undirected graphical model, it is the set of neighbors of $B$, while in directed graphical models, it is the set of parents, co-parents, and children of $B$.

Our generalizations of Markov blankets are minimal d-separators between two sets $B$ and $D$.
As we mentioned in Example~\ref{ex:causality}, minimal d-separators play a role in causality theory.
 In that context, \citet{Tian98findingminimal} prove that a minimal d-separator between two subsets of variables can be found with a polynomial algorithm in $O(\vert V \vert .\vert A \vert)$. 


\section{Markov blanket in a set}
\label{sec:markov_blanket_in_a_set}
We now introduce the notion of Markov blanket in a set.
\begin{de}\label{de:markov_blanket_in_a_set}
      Let $B$, $C$ and $E$ be three set of vertices in a graph $G = (V,A)$. The \emph{Markov blanket of $B$ in $C$ given E}, denoted by $\mb_C(B | E)$, is the smallest subset $M \subseteq C$ of vertices satisfying 
      \begin{equation}\label{eq:mbCdef}
      	X_B \indep X_{C\backslash (B \cup M)} | X_{M \cup E} \quad \text{for any distribution that factorizes on $G$,}
      \end{equation}
      where smallest means that a set $M\subseteq C$ satisfies \eqref{eq:mbCdef} if and only if $\mb_C(B|E) \subseteq M$.
\end{de}

\noindent Note that this definition holds both in directed and undirected graphical model. When $E= \emptyset$, we use the simpler notation $\mb_C(B)$.  The Markov blanket $\mb_C(B)$ coincides with $\mb(B)$ if $C = V$. 
Figure~\ref{fig:markovBlanket} illustrates the difference between the usual Markov blanket and the Markov blanket in a set. 

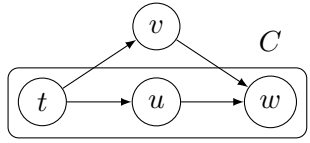
\begin{figure}\label{fig:markov_blanket_in_a_set}
      \begin{center}
            \begin{tikzpicture}
                  \def\l{1.5}
                  \def\h{1}
                  \node[sta] (t) at (0*\l,0*\h) {$t$};
                  \node[sta] (u) at (1*\l,0*\h) {$u$};
                  \node[sta] (v) at (1*\l,1*\h) {$v$};
                  \node[sta] (w) at (2*\l,0*\h) {$w$};
                  \draw[arc] (t) -- (u);
                  \draw[arc] (t) -- (v);
                  \draw[arc] (u) -- (w);
                  \draw[arc] (v) -- (w);
                  \node[draw, rounded corners, fit=(t.west) (t.north) (w.east) (w.south)] (C) {};
                  \node at (2*\l, 0.8) {$C$};
            \end{tikzpicture}
      \end{center}
      \caption{The Markov blanket of $t$ is $\{u,v\}$, and its Markov blanket in $C$ is $\{u,w\}$.}
      \label{fig:markovBlanket}
\end{figure}
The next theorem shows the existence and uniqueness of the Markov Blanket in a set and provides a graphical characterization in directed and undirected graphical models.
\begin{theo}\label{prop:mbInCdefinition}
      Let $B$, $C$ and $E$ be three sets of vertices in a graph $G = (V,A)$. The Markov blanket of $B$ in $C$ given $E$ exists, is unique, and equal to 
      \begin{equation}\label{eq:MB_def}
            \mb_C(B|E) = \Big\{v \in C \colon v\text{ is not (d-)separated from $B$ given } E \cup \big(C\backslash (B \cup \{v\})\big) \Big\},
      \end{equation}
      \noindent where ``d-separated'' and ``separated'' apply in directed and undirected graphical models respectively.
\end{theo}

\noindent  
The Markov blanket in a set no longer admits a characterization in terms of parents, coparents, children and neighbor vertices. 
However, thanks to the characterization~\ref{eq:MB_def}, $\mb_C(B|E)$ can be computed in $O\big(|C|(|A|+|V|)\big)$ using a (d-)separation algorithm (\citet{Geiger1990dseparation}).
\begin{proof}[Proof of Theorem~\ref{prop:mbInCdefinition}, undirected graphical models]
Let $B$, $C$, and $E$ be three sets of vertices, and $M$ as in \eqref{eq:MB_def}.

We start by proving that $B$ is separated from ${C\backslash (B \cup M)}$ given  $M\cup E$. 
Let $v$ be a vertex in $C \backslash (B \cup M)$, and $P$ be a $B$-$v$ path. 
As $v$ does not belong to $M$, path $P$ is not active given $E \cup \big(C\backslash (B \cup \{v\})\big)$, and there is a vertex in $E \cup \big(C\backslash (B \cup \{v\})\big)$ on $P\backslash \{v\}$. 
Let $w$ be the first vertex of $P$ in that set, starting from $B$. 
If $w$ is in $E$, path $P$ is not active given $E\cup M$. 
Otherwise, the $B$-$w$ restriction of $P$ is active given $E \cup \big(C\backslash (B \cup \{w\})\big)$. 
Vertex $w$ thus belongs to $M$ and $P$ is not active given $E\cup M$, which gives the result.

Let $N$ be a subset of $C$ such that $B$ is separated from $C \backslash (N \cup B)$ given $N\cup E$. Let $v$ be a vertex in $M$. 
By definition of $v$, there exists a $B$-v path that is active given $E \cup C \backslash(B \cup \{v\})$ with a minimum number of arcs. 
Let $P$ be such a path. 
The only intersection of $P$ with $E \cup C$ is \{v\}. 
Path $P$ is therefore not active given $E\cup N$ if and only if $v$ belongs to $N$. 
Hence $v$ belongs to $N$, and we obtain $M\subseteq N$.  
\end{proof}

The proof for directed graphical models is similar but more technical due to d-separation.

\begin{proof}[Proof of Theorem~\ref{prop:mbInCdefinition}, directed graphical models]

      Let $B$, $C$, and $E$ be three sets of vertices, and $M$ as in \eqref{eq:MB_def}.

      We start by proving that $B$ is d-separated from ${C\backslash (B \cup M)}$ given  $M\cup E$. Let $P$ be a trail between a vertex $b\in B$ and a vertex $v \in C\backslash (B \cup M)$.
      We prove that $P$ is not active.
      Without loss of generality, we can suppose that $P\cap B = \{b\}$.
      Indeed, if $P$ is active, then any of its subtrails whose extremities are not in $M$ must be active. 
      As $B \cap M = \emptyset$, it suffices to show that the subtrail $Q$ between the last  vertex of $P$ in $B$ (starting from $b$) is not active.
      If $P$ has a v-structure that is not active given $E \cup M$, or if $P$ has a vertex that is not a v-structure in $E \cup M$, then $P$ is not active. 
      Suppose now that we are not in one of those cases.
      Starting from $b$, let $w$ bet the first vertex of $P$ in $C$ 
      that is not the middle of a v-structure in $P$, and
      let $Q$ be the $b$-$w$ subtrail of $P$.
      By definition of $w$, any vertex of $Q$ that is not in the middle of a $v$-structure is not in $C$, and by hypothesis it is not in $E$, hence it is not in $E \cup \left( C\backslash(B\cup \{w\})\right)$. 
      Furthermore, by hypothesis, any v-structure of $Q$ is active given $E \cup M$.
      Suppose that $w$ is not in $M$: we obtain $M \subseteq E \cup (C\backslash(B\cup \{w\}))$, and hence, any v-structure of $Q$ is active given $E \cup (C\backslash(B\cup \{w\}))$. Therefore $Q$ is active given $E \cup (C\backslash(B\cup \{w\}))$ and $w \in M$, which is a contradiction. We deduce that $w \in M$.
      Hence $w \neq v$.
      As $w \in M$ is not in the middle of a v-structure, $P$ is not active given $M \cup E$, which gives the result.

      \begin{figure}\label{fig:proofDefinitionMbC}
            \begin{center}

                  \begin{tikzpicture}

                        \node[sta] (b ) at (0,1) {$b $};
                        \node[sta] (v1) at (1,2) {};
                        \node[sta] (u1) at (2,0) {$u_1$};
                        \node[sta] (v2) at (3,3) {};
                        \node[sta] (ui) at (4,2.3) {$u_i$};
                        \node[sta] (v3) at (5,3) {};
                        \node[sta] (uj) at (6,0) {$u_j$};
                        \node[sta] (v ) at (7,1) {$v $};
                        \node[sta] (w') at (4,0.9) {$w'$};
                        \node[sta] (w) at (4,-1.1) {$w$};
                        \node[sta] (a) at (-1, 0) {};
                        \node[sta] (z) at (8, 0) {};

                        \draw[dashed,arc,color=purple] (v1) -- (u1);
                        \draw[dashed,arc,color=purple] (v2) to node[midway,above left]{$P$} (u1);
                        \draw[dashed,arc,color=purple] (v2) -- (ui);
                        \draw[dashed,arc,color=purple] (v3) -- (ui);
                        \draw[dashed,arc,color=purple] (v3) -- (uj);
                        \draw[dashed,arc,color=purple] (v ) -- (uj);
                        \draw[dashed,arc,color=purple] (v1) -- (b );
                        \draw[dashed,arc] (ui) -- (w');
                        \draw[dashed,arc] (w') -- (w);


                        \node[draw, inner sep=2mm,label=below left:$C$,fit=(b .north) (w.south) (a.west) (z.east)] {};

                        \node[fill=blue, opacity=0.1, inner sep=2mm,fit=(w'.north) (u1.south) (u1.west) (v .east)] {};
                        \node[draw,inner sep=2mm,label=below right:$M$,fit=(w'.north) (u1.south) (u1.west) (v .east)] {};

                        \node[fill=red, opacity=0.1,inner sep=2mm,label=below:$B$,fit=(b .north) (a.south) (a.west) (b .east)] {};
                        \node[draw,inner sep=2mm,label=below:$B$,fit=(b .north) (a.south) (a.west) (b .east)] {};

                  \end{tikzpicture}
            \end{center}
      \caption{Illustration of the proof of Theorem~\ref{prop:mbInCdefinition}}
      \label{fig:proofDefinitionMbC}
      \end{figure}
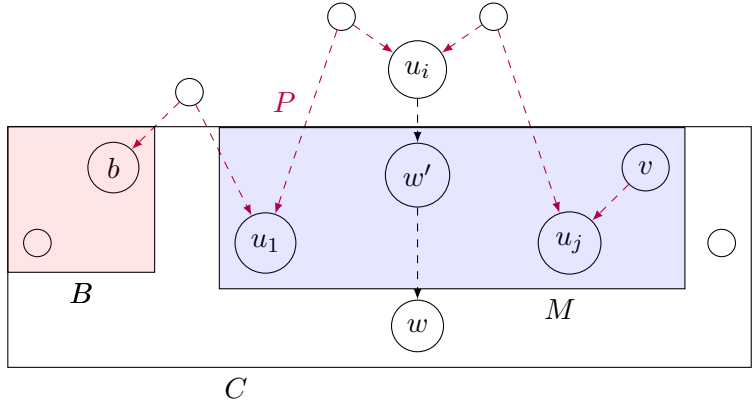

      Conversely, let $N\subseteq C$ be a set of vertices such that $B$ is d-separated from $C\backslash(N\cup B)$ given $N\cup E$.
      We now prove that $M \subseteq N$.
      This part of the proof is illustrated on Figure~\ref{fig:proofDefinitionMbC}.
      Let $v$ be a vertex in $M$. 
      As $v$ is in $M$, there is an active trail between $B$ and $v$ given $E \cup \big(C \backslash(\{v\}\cup B)\big)$. Let $P$ be such a trail.
      Without loss of generality, we can suppose $B \cap P = \{b\}$.
      As $P$ is active given $E \cup \big(C \backslash(\{v\}\cup B)\big)$ and $B \cap P = \{b\}$, any vertex of $P\backslash \{b,v\}$ that is not in the middle of a v-structure is not in $C \backslash(\{v\}\cup B)$, and hence not in $C$, and not in $N$.
      Starting from $b$, let $u_1,\ldots, u_k$ be an indexation of the vertices of $P$ that are in the middle of $v$-structures in $P$.
      We prove by iteration on $i$ that $\cdsc{u_i} \cap \big(E \cup N\big) \neq \emptyset$. Suppose the result true up to $i-1$, and $P_i$ be the subtrail of $P$ from $b$ to $u_i$.
      Suppose that $u_i$ is not in $\casc{E}$.
      As $P$ is an active trail given $E \cup \big(C \backslash(\{v\}\cup B)\big)$ and $u_i$ is in the middle of a v-structure, $u_i$ has a descendant $w$ in $C \backslash(\{v\}\cup B)$, and there is a directed path $Q$ from $u_i$ to $w$.
      Let $w'$ be the first vertex of $Q$ in $C\backslash (\{v\}\cup B)$ and $Q'$ the $u_i$-$w'$ restriction of $Q$. 
      Note that we may have $u_i = w$ or $u_i = w'$.
      Suppose that $w' \notin N$. It implies that $w' \in C \backslash (N \cup E)$. By induction hypothesis, the trail $P_i$ followed by Q' is active given $N \cup E$ between $B$ and $C \backslash (N \cup E)$. It contradicts Equation~\eqref{eq:mbCdef} for $N$. We deduce that $w' \in N$.
      Finally, as any vertex of $P\backslash\{b,v\}$ that is not in the middle of a v-structure is not in $N$, and $\cdsc{u} \cap N \neq \emptyset$ for any vertex $u$ of $P$ that is in the middle of a v-structure, the path $P$ is not active given $N$ only if $v \in N$.
      As $B$ is d-separated from $C\backslash(N\cup B)$ given $N$, we have $v\in N$, which gives the result, and the first part of the proposition.

      It is then an immediate corollary that any set $M \subseteq C\backslash B$ containing $\mb_C(B|E)$ satisfies Equation~\eqref{eq:mbCdef}.
\end{proof}

Theorem~\ref{prop:mbInCdefinition} ensures that $C' \perp B | C \cup E$ if and only if $\mb_{C\cup C'}(B|E) \subseteq C$. The following proposition strengthens this result.

\begin{prop}\label{prop:mbInCcupIndep} 
Let $B$, $C$, $C'$ and $E$ be four sets of vertices. Then $\mb_{C \cup C'}(B|E) = \mb_C(B|E)$ if and only if $C' \perp B | C \cup E$.
\end{prop}

From Definition~\ref{de:markov_blanket_in_a_set}, it is clear that $\mb_{C \cup C'}(B|E) = \mb_C(B|E)$ implies $C' \perp B | C \cup E$, and that $C' \perp B | C \cup E$ implies $\mb_{C \cup C'}(B|E) \subseteq C$. So we only have to show that $C' \perp B | C \cup E$ implies $\mb_{C \cup C'}(B|E) = \mb_C(B|E)$.

\begin{proof}[Proof of Proposition \ref{prop:mbInCcupIndep} for undirected graphical models]
      Suppose that $C' \perp B | C \cup E$. Let $v \in \mb_{C \cup C'}(B|E)$, there exists an active path Q between $B$ and $v$ such that $Q \cap (C \cup C' \cup E) = \emptyset$. Therefore $Q \cap (C \cup E) = \emptyset$. If $v \in C'$, then the assumption $C' \perp B | C \cup E$ implies that $Q \cap (C \cup E) \neq \emptyset$, which contradicts our assumption. We deduce that $v \in C$ and $v$ is not separated from $B$ by $C \cup E$. Therefore, $v \in \mb_{C}(B|E)$.
      Let $u \in \mb_{C}(B|E)$, there exists a path Q from $B$ to $u$ such that $Q \cap (C \cup E) = \emptyset$. If $Q \cap C' \neq \emptyset$, the assumption $C' \perp B | C \cup E$ implies that $C \cup E$ intersects Q which contradicts our assumption on Q. Therefore, $Q \cap (C \cup C' \cup E) = \emptyset$. We deduce that $v \in \mb_{C \cup C'}(B|E)$.
      It achieves the proof.
\end{proof}

\begin{proof}[Proof of Proposition \ref{prop:mbInCcupIndep} for directed graphical models]

      Let $B$, $C$, and $C'$ be such that $C' \perp B | C \cup E$. 
      We only have to show that, given a vertex $v$ in $C$ and a $B$-$v$ trail $P$, then $P$ is active given $\big(C \cup C') \backslash (B \cup \{v\})\big) \cup E$ if and only if $P$ is active given $(C \backslash (B \cup \{v\})) \cup E$.
      Let $v$ be a vertex in $C$ and $P$ be a $B$-$v$ trail. 
      W.l.o.g., we suppose that it intersects $B$ at most once, and $v$ at most once.
      
      Suppose that $P$ is active given $(C \backslash (B \cup \{v\})) \cup E$.
      Then $P$ does not intersect $C'$. Indeed, suppose it intersects $C'$ in a vertex $w$.
      Then, the $B$-$w$ subtrail is active given $C \backslash (B \cup \{v\}) \cup E$, which contradicts $B \perp C' | C \cup E$.
      Furthermore, all the v-structures of $P$ are active given $\big(C \cup C') \backslash (B \cup \{v\})\big) \cup E$, as they have a descendant in $(C \backslash (B \cup \{v\})) \cup E$. Hence $P$ is active given $\big(C \cup C') \backslash (B \cup \{v\})\big) \cup E$. 


      Suppose now that $P$ is active given $\big(C \cup C') \backslash (B \cup \{v\})\big) \cup E$.
      It intersects $C \backslash (B \cup \{v\})) \cup E$ only on v-structures, and all these $v$-structures are active given $\big(C \cup C') \backslash (B \cup \{v\})\big) \cup E$. 
      Suppose that there is a v-structure that is not active given $(C \backslash (B \cup \{v\})) \cup E$, and let $s$ be the first one starting from $B$. Then $s$ has a descendant $w$ in $C' \backslash (C \cup E)$, and the $B$-$s$ subtrail of $P$ followed by the $s$-$w$ path is active given $C \cup E$, which contradicts $B\perp C' | C \cup E$. 
      Hence $P$ is active given $(C \cup C' \backslash (B \cup \{v\})) \cup E$.
\end{proof}

\section{Directional Markov blanket}
\label{sec:directional_markov_blanket}

We write ``a (d-)separator $S$'' when we make statement that hold both in directed and undirected graphical models. 
Set $S$ is a then a d-separator in directed graphical models, and a separator in undirected graphical models.

\begin{de}\label{de:MB_directional}
	Let $B$,$D$, and $E$ be three sets of vertices in a graph $G = (V,A)$. The \emph{Markov blanket of $B$ in the direction of $D$ given $E$}, denoted $\mb(B \rightarrow D | E)$, is the minimal (d-)separator $M$ of $B$ and $D$ such that 
	\begin{equation}\label{eq:nearestDseparator}
		D \perp M | M' \cup E \quad \text{for any (d-)separator $M'$ between $B$ and $D$ given $E$.}
	\end{equation}
\end{de}

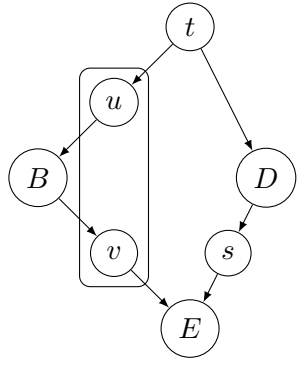
\begin{figure}
	\begin{center}
		\begin{tikzpicture}
			\node[sta] (B) at (0,2) {$B$};
			\node[sta] (v) at (1,1) {$v$};
			\node[sta] (E) at (2,0) {$E$};
			\node[sta] (s) at (2.5,1) {$s$};
			\node[sta] (D) at (3,2) {$D$};
			\node[sta] (u) at (1,3) {$u$};
			\node[sta] (t) at (2,4) {$t$};
			\draw[arc] (u) -- (B);
			\draw[arc] (t) -- (u);
			\draw[arc] (t) -- (D);
			\draw[arc] (B) -- (v);
			\draw[arc] (v) -- (E);
			\draw[arc] (D) -- (s);
			\draw[arc] (s) -- (E);
			\node[draw, rounded corners, fit=(u.west) (u.north) (u.east) (v.south)] (M) {};
			\node (M) at (0.3,0.2) {};
		\end{tikzpicture}
	\caption{Example of the directional Markov Blanket from $B$ to $D$ given an evidence set $E$. In this case $\mb(B\rightarrow D \vert E) = \{u, v\}$}
	\label{fig:example_MB_evidence}
	\end{center}
\end{figure}

\noindent Figure \ref{fig:example_MB_evidence} shows an example of the Markov Blanket of $B$ in the direction $D$ given an evidence $E$. Note that in this definition, the evidence set $E$ can be such that $E \cap B \neq \emptyset$.
\noindent The Markov blanket of $B$ in the direction of $D$ is the d-separator between $B$ and $D$ that is the nearest to $B$. Furthermore, the following proposition provides an alternative definition.
\begin{prop}\label{prop:equivalence_def_directionalMB}
	Let $B$,$D$, and $E$ be three sets of vertices in a graph $G=(V,A)$. Let $M$ be a minimal (d-)separator between $B$ and $D$ given $E$.

	$M$ satisfies \eqref{eq:nearestDseparator} if, and only if :
	\begin{equation}\label{eq:nearestDseparator_newversion}
		B \perp M' | M \cup E \quad \text{for any  minimal (d-)separator $M'$ between $B$ and $D$ given $E$.} 
	\end{equation}
\end{prop}

\begin{proof}[Proof of Proposition \ref{prop:equivalence_def_directionalMB} in undirected graphical models]
	Let $M$ be a minimal d-separator.

	We start by proving \eqref{eq:nearestDseparator} implies \eqref{eq:nearestDseparator_newversion}. L
	Let $M'$ be a minimal separator between $B$ and $D$ given $E$, and let P be a path between $B$ and $x \in M'$, where  Since $M'$ is minimal, there exists a path $Q$ from $x$ to $D$ such that $Q \cap (M' \cup E) \backslash \{x\} = \emptyset$. The path $R$ composed of $P$ followed by $Q$ is a $B$-$D$ path. Since $M$ is a d-separator, there exists $v \in R \cap M$. If $v \in Q$, then \eqref{eq:nearestDseparator} implies that $Q \cap (M' \backslash \{x\}) \neq \emptyset$, which contradicts the assumption on $Q$. Therefore, $v \in P$. We deduce that all path from $B$ to $M'$ is intersected by $M \cup E$, which implies that $B \perp M' | M \cup E$.

	Suppose now that \eqref{eq:nearestDseparator_newversion} holds. Let $Q$ be a path from $u \in M$ to $D$ and $M'$ be a separator between $B$ and $D$ given $E$. Since $M$ is minimal, there exists a path $P$ from $B$ to $u$ such that $P \cap (M \backslash \{u\}) = \emptyset$. The path $R$ composed of $P$ followed by $Q$ is a $B$-$D$ path, there exists $v \in R \cap (M' \cup E)$. Using the same arguments as above, $v \in Q$, which implies that $x \perp D | M' \cup E$.
\end{proof}

The proof of Proposition~\ref{prop:equivalence_def_directionalMB} in directed graphical model is more involved and postponed to Section~\ref{sub:proof_of_alternative_definition}.
Similarly to the Markov Blanket in a set, we need to prove that $\mb(B \rightarrow D |E)$ in Definition \ref{de:MB_directional} exists. The following theorem states the existence and uniqueness of the Directional Markov Blanket.

\begin{theo}\label{theo:MB_directional_closed_formula}
Let $B$,$D$, and $E$ be three sets of vertices in a graph $G=(V,A)$. If there exists a (d-)separator between $B$ and $D$ given $E$, the Markov blanket of $B$ in the direction of $D$ given $E$ exists, is unique, and is given by
\begin{align*}
\mb(B\rightarrow D \vert E) &= \mb_{\mb(B|E)}(D | E) &&\text{in undirected graphical models, and by} \\
\mb(B\rightarrow D \vert E) &= \mb_{\mb_{\casc{B\cup D \cup E}}(B|E)}(D | E) && \text{in directed graphical models.}
\end{align*}
\end{theo}

The rest of the section is dedicated to the proofs of Proposition~\ref{prop:equivalence_def_directionalMB} in directed graphical models and of Theorem~\ref{theo:MB_directional_closed_formula}.

\begin{rem} 
Using Definition~\ref{de:MB_directional}, the Markov blanket of $B$ in the direction of $D$ given $E$ exists if and only if there exists a d-separator between $B$ and $D$ given $E$.
We can extend the definition of the Markov blanket of $B$ in the direction of $D$ given $E$ as the set $M$ of $V\backslash B$ satisfying 
\begin{enumerate}[label = (\roman*)]
            \item \label{item:sep} $B\perp D | M\cup E$, \hfill 
            \item \label{item:nearest} $B \perp D | M' \cup E$ implies $D \perp M | M'\cup E$,
            \item \label{item:smallest} $B \perp D | M' \cup E$ and $D \perp M' | M \cup E$ implies $M \subseteq M'$. 
\end{enumerate}	
It is immediate that the two definitions coincide when there exists a d-separator between $B$ and $D$ given $E$. 
But this alternative definition does not require the existence of a d-separator between $B$ and $D$. 
With this new definition, even without the existence of a d-separator, it follows from Theorem~\ref{theo:characterization_mbdirectional} in Section~\ref{sub:proof_of_the_characterization} that $\mb(B\rightarrow D | E)$ exists and admits the following updated characterization
\begin{align*}
\mb(B\rightarrow D \vert E) &= \depd{D} \cup \mb_{\mb(B|E)}(\indd{D} | E) &&\text{in undirected graphical models, and by} \\
\mb(B\rightarrow D \vert E) &= \depd{D} \cup  \mb_{\mb_{\casc{B\cup D \cup E}}(B|E)}(\indd{D} | E) && \text{in directed graphical models,}
\end{align*}
where 
$$\depd{D} = \left\{ \begin{array}{ll}
D \cap \mb(B|E) &\text{in undirected graphical models,} \\
D \cap \mb_{\casc{B \cup D \cup E}}(B|E)\enskip& \text{in directed graphical models,}
\end{array}\right.
$$
and $\indd{D} = D\backslash \depd{D}$.
\end{rem}

\subsection{Preliminary lemmas in directed graphical models}
\label{sub:proofs}

In this section we present some technical results on d-separators in directed graphical models. In the remaining of this section $B$, $D$ and $E$ denote three sets of vertices in a graph $G = (V,A)$.

\begin{lem}\label{lem:BDtrailintersectDseparatingSet}
Let $M$ be a d-separator between $B$ and $D$ given $E$. Then any $B$-$D$ trail in $\casc{B \cup D \cup M \cup E}$ intersects $M \cup E$ in a vertex $x$ that is not a v-structure.
\end{lem}

\begin{proof} 
Let $P$ be a $B$-$D$ trail in $\casc{B \cup D \cup M \cup E}$. Starting from $B$, let $v$  be the last v-structure of $P$ that is not active given $M \cup E$ and that is in $\asc{B}$, with $v$ being equal to the first vertex of $P$ if there is no such v-structure.
Starting from $v$, 
let $w$ be equal to the first v-structure of the $v$-$d$ subpath of $P$ that is not active given $M \cup E$, and to the last vertex of $P$ if there is no such v-structure. By definition of $v$, vertex $w$ has necessarily a descendant in $D$. Taking a $B$-$w$ path followed by the $v$-$w$ subtrail of $P$ and then a $w$-$D$ path, we obtain an active trail given $M \cup E$, which gives a contradiction. 
\end{proof}

\begin{lem}\label{lem:addingVerticesDseparator}
Let $M$ be a d-separator between $B$ and $D$ given $E$, and $N \subseteq \casc{B \cup D \cup M \cup E}$. Then $M \cup N$ is a d-separator between $B$ and $D$ given $E$.  
\end{lem}
\begin{proof}
Suppose that there exists an active trail between $B$ and $D$ given $M \cup E \cup N$. Let $P$ be such a trail. Since $N \in \casc{B \cup D \cup M \cup E}$, we deduce that $P$ is a trail in $\casc{B \cup D \cup M \cup E}$ because all v-structures have a descendant in $M \cup E \cup N$ and $N \subset \casc{B \cup D \cup M \cup E}$. Lemma \ref{lem:BDtrailintersectDseparatingSet} ensures that $P$ intersects $M \cup E$ in a vertex that is not a v-structure. It contradicts the assumption on $P$.
\end{proof}

The following lemma is an extension of Theorem 6 of \citet{Tian98findingminimal} where we allow an evidence $E$.

\begin{lem}\label{lem:IntersectionWithAscendantsIsStillDseparating}
If $M$ is a d-separator between $B$ and $D$ given $E$, then $M \cap \casc{B \cup D \cup E}$ is also a d-separator between $B$ and $D$ given $E$.
\end{lem}
\begin{proof}
Any trail that intersects $V \backslash \casc{B \cup D \cup E}$ is not active given $(M \cap \casc{B \cup D \cup E}) \cup E$.
And by Lemma~\ref{lem:BDtrailintersectDseparatingSet}, any trail in $\casc{B \cup D \cup E}$ intersects $(M \cap \casc{B \cup D \cup E}) \cup E$ on a non v-structure, which gives the result.
\end{proof}

\begin{coro}\label{coro:existanceOfADSeparatingSubset}
Let $M$ be a set of vertices. Then there exists a subset of $M$ that d-separates $B$ and $D$ given $E$ if and only if 
$$ B \perp D \big| (M \cap \casc{B\cup D \cup E}) \cup E$$

\end{coro}
\begin{proof}
An immediate corollary of the two previous lemmas.
\end{proof}

\begin{lem}\label{lem:allVerticesDseparatedByM}
Let $M$ be a d-separator between $B$ and $D$ given $E$, and $x \in \casc{B \cup D \cup M \cup E}$. Then at least one of the following statement is true: $x \perp B | M \cup E$ or $x \perp D | M \cup E$.
\end{lem}
\begin{proof}
Suppose that none of the independences are satisfied. Then $x\notin M$, and there is a $B$-$x$  trail $Q$ that is active given $M \cup E$, and an $x$-$D$ trail $R$ that is active given $M \cup E$. As $x \in \casc{B\cup D \cup M \cup E}$, if trails $Q$ and $R$ intersect $V \backslash(\casc{B\cup D \cup M \cup E})$, they are not active given $M \cup E$. As $x \notin M \cup E$, the trail composed of $Q$ followed by $R$ is a $B$-$D$ trail that intersects $M \cup E$ only on v-structures. This contradicts Lemma~\ref{lem:BDtrailintersectDseparatingSet}, and gives the result.
\end{proof}

\subsection{Proof of Theorem~\ref{theo:MB_directional_closed_formula}}
\label{sub:proof_of_the_characterization}

In this section we prove Theorem~\ref{theo:MB_directional_closed_formula}.

\begin{lem}\label{lem:mbset_dsep} 
Let $M$ be a (d-)separator between $B$ and $D$ given $E$, then $\mb_{M}(B|E)$ is a (d-)sepa\-rator between $B$ and $D$ given $E$.
\end{lem}

\begin{proof}[Proof of Lemma \ref{lem:mbset_dsep} in undirected graphical models]
      Consider a path Q from $B$ to $D$. Since $B \perp D | M \cup E$, we have $Q \cap (M \cup E) \neq \emptyset$. Starting from $B$, consider the first vertex $x$ of $M \cup E$ on the path Q. 
      By Theorem~\ref{prop:mbInCdefinition}, $x \in \mb_{M}(B|E)$. It implies that $Q \cap \mb_{M}(B|E) \neq \emptyset$. We conclude that $B$ and $D$ are separated by $\mb_{M}(B|E) \cup E$.
\end{proof}

\begin{proof}[Proof of Lemma \ref{lem:mbset_dsep} in directed graphical models]
      Suppose that $B \nperp D | \mb_{M}(B|E) \cup E$. 
      Let $P$ be a trail between $B$ and $D$ that is active given $\mb_{M}(B|E) \cup E$.
      Since $\mb_{M}(B|E) \cup E \subseteq M \cup$, all the v-structures of $P$ are active given $M\cup E$.
      Since $P$ is not active given $M \cup E$, there exists at least one element in $(M \cup E) \cap P$, which is not in a v-structure of $P$. Starting from $B$, consider the first element $x$ on P such that $x \in (M \backslash \{x\})  \cup E$. The subtrail of $P$ from $B$ to $x$ is active given $(M \backslash \{x\}) \cup E$. Therefore, $x \in \mb_{M}(B|E)$, which contradicts our assumption on P. 
\end{proof}

\begin{coro}\label{coro:mbset_dsep_min}
      Let $M$ be a minimal (d-)separator between $B$ and $D$, then $\mb_{M}(B|E) = M$.
\end{coro}

\begin{proof}
      Lemma \ref{lem:mbset_dsep} ensures that $\mb_{M}(B|E)$ is a d-separator (resp.~separator) between $B$ and $D$ given $E$. Since $\mb_{M}(B|E) \subseteq M$ and $M$ is minimal, we deduce that $\mb_{M}(B|E) = M$.
\end{proof}

\begin{lem}\label{lem:minimalDseparatingSetInM}
Let $B$ and $D$ given $E$ be three sets of vertices of an undirected graphical model (resp.~directed graphical model) $G= (V,E)$. 
Let $M$ be a separator between $B$ and $D$ given $E$ (resp.~a d-separator between $B$ and $D$ given $E$ in $\casc{B \cup D \cup E}$). 
Then $\mb_M(B|E)$ is a (d-)separator between $B$ and $D$ given $E$, and $\mb_{\mb_M(B|E)}(D|E)$ is a minimal (d-)separator between $B$ and $D$ given $E$.
\end{lem}

\begin{proof}[Proof of Lemma \ref{lem:minimalDseparatingSetInM} in undirected graphical models]
	Let $M' = \mb_M(B|E)$ and $M''$ be equal to $\mb_{\mb_M(B|E)}(D|E)$. Lemma \ref{lem:mbset_dsep} ensures that $M'$ and $M''$ are separators between $B$ and $D$ given $E$. We prove that $M''$ is minimal. Let $v \in M''$. There exists a path $P$ from $B$ to $v$ such that $P \cap (M \cup E) \backslash \{x \} = \emptyset$ and there exists a path $Q$ from $v$ to $D$ such that $Q \cap (M' \cup E) \backslash \{x\} = \emptyset$. Consider the path R composed of $P$ followed by $Q$. Then $R$ is a $B$-$D$ path with $R \cap (M'' \cup E \backslash \{v\} ) = \emptyset$. We deduce that $R$ is not separated by $M'' \backslash \{v\} \cup E$, which implies that $M''\backslash \{v\}$ is not a separator given $E$. It achieves the proof.
\end{proof}

\begin{proof}[Proof of Lemma \ref{lem:minimalDseparatingSetInM} in directed graphical models]  
	Let $M' = \mb_M(B|E)$ and $M''$ be defined as $\mb_{\mb_M(B|E)}(D|E)$. We prove that $M'' = \mb_{\mb_M(B|E)}(D|E)$ is a minimal d-separator. Lemma \ref{lem:mbset_dsep} ensures that $M'$ and $M''$ are d-separators between $B$ and $D$ given $E$.
	Let $v$ be a vertex in $M''$.
	Let $Q$ be a $B$-$v$ trail active given $M \cup E\backslash \{v\}$, and $R$ be a $v$-$D$ trail active given $M' \cup E \backslash \{v\}$, and $P$ the trail composed of $R$ followed by $Q$. Then $P$ is a $B$-$D$ trail in $\casc{B \cup D \cup E}$ that intersects $M'' \cup E \backslash \{v\}$ only on v-structures. 
	Hence, Lemma~\ref{lem:BDtrailintersectDseparatingSet} ensures that $M'' \backslash \{v\}$ is not a d-separator, and Corollary~\ref{coro:existanceOfADSeparatingSubset} enables to conclude that $M''$ is a minimal d-separator.
\end{proof}

The following theorem is a stronger version of Theorem~\ref{theo:MB_directional_closed_formula}.

\begin{theo}\label{theo:characterization_mbdirectional}
Let $B$ and $D$ given $E$ be three sets of vertices of an undirected graphical model (resp.~directed graphical model) $G= (V,E)$. 
Let $M$ be a separator between $B$ and $D$ given $E$ (resp.~a d-separator between $B$ and $D$ given $E$ in $\casc{B \cup D \cup E}$). 
Then $M_1 = \mb_{\mb_M(B|E)}(D|E)$ is the unique minimal (d-)separator between $B$ and $D$ given $E$ such that $ M_1 \perp D | M_2 \cup E$ for any (d-)separator $M_2$ in $M$.
\end{theo}

\begin{proof}[Proof of uniqueness in Theorem~\ref{theo:characterization_mbdirectional}]
Suppose that $M_1$ and $M'_1$ are minimal (d-)separator between $B$ and $D$ given $E$ such that $ M_1 \perp D | M_2 \cup E$ for any (d-)separator $M_2$ in $M$.
Then $M'_1 \perp D | M_1$ gives $\mb_{M_1 \cup M_1'}(D|E) \subseteq M_1$.
As $M_1$ is a minimal d-separator, Corollary~\ref{coro:mbset_dsep_min} gives $\mb_{M_1}(D|E) = M_1$, and we deduce $\mb_{M_1 \cup M_1'}(D|E) = M_1$.
Exchanging the roles of $M_1$ and $M'_1$ gives $\mb_{M_1 \cup M_1'}(D|E) = M'_1$, and we obtain $M_1=M'_1$ and the uniqueness result.
\end{proof}


\begin{proof}[Proof of Theorem \ref{theo:characterization_mbdirectional} in undirected graphical models]
	Lemma~\ref{lem:minimalDseparatingSetInM} ensures that $M_1$ is a minimal separator between $B$ and $D$ given $E$.
	Let $M_2 \subseteq M$ be a separator between $B$ and $D$ given $E$. We prove that $M_1 \perp D | M_2 \cup E$. Suppose that $M_1 \nperp D | M_2 \cup E$.
	There exists an active path between $v \in M_1$ and $D$ given $M_2 \cup E$. Let $Q$ be such a path. Therefore we have $Q \cap (M_2 \cup E) = \emptyset$. Since $v \in \mb_{M}(B|E)$, there exists an active path between $B$ and $v$ given $M \cup E \backslash \{v\}$. Let $P$ be such a path. Therefore we have $P \cap (M \cup E) = \emptyset$. Let $R$ be the path composed of $P$ followed by $Q$. $R$ is a $B$-$D$ path and $R \cap (M_2 \cup E) =\emptyset$, which contradicts the assumption on $M_2$. 
\end{proof}

\begin{proof}[Proof of Theorem \ref{theo:characterization_mbdirectional} in directed graphical models]
	Lemma~\ref{lem:minimalDseparatingSetInM} ensures that $M_1$ is a minimal d-separator between $B$ and $D$ given $E$.
	Let $M_2 \subseteq M$ be a d-separator between $B$ and $D$ given $E$. We prove $M_1 \perp D | M_2 \cup E$. Suppose that $M_1 \nperp D | M_2 \cup E$, there exists an active trail between $v \in M_1$ and $D$ given $M_2 \cup E$. Let $Q$ be such a trail. Since $v \in \mb_{M}(B|E)$, there exists an active trail from $B$ to $v$ given $M \cup E \backslash \{v\}$. Let $P$ be such a trail and $R$ be the trail composed of $P$ followed by $Q$. $R$ is a trail in $\casc{B \cup D \cup M_2 \cup E}$ and $M_2 \cup E$ intersects $R$ only on v-structures. Lemma \ref{lem:BDtrailintersectDseparatingSet} ensures that $M_2$ is not a d-separator between $B$ and $D$ given $E$, which contradicts the assumption on $M_2$.
\end{proof}

\subsection{Proof of Proposition~\ref{prop:equivalence_def_directionalMB} in directed graphical models}
\label{sub:proof_of_alternative_definition}

The two following lemmas are intermediary technical results for the proof of the alternative definition of the directional Markov Blanket in directed graphical models in Proposition \ref{prop:equivalence_def_directionalMB}.

\begin{lem} \label{lem:elementOfMinimalIsDseparatedByBlanket}
Let $M$ be a minimal d-separator between $B$ and $D$ given $E$. 
Let $N \subseteq \casc{B \cup D \cup E}$. 
Let $L = \mb_{M \cup N}(B|E)$, and $O = \mb_{L}(D|E)$.
Then $L \cap M = O \cap M$.
\end{lem}

\begin{proof}
Remark that $M \subseteq \casc{B \cup D \cup E}$ because $M$ is a minimal d-separator between $B$ and $D$ given $E$. Inclusion $O \subseteq L$ gives $O \cap M \subseteq L \cap M$. 
Suppose that $O \cap M \neq L \cap M$. Since $L \cap M $ contains strictly $O \cap M$, it ensures the existence of $x$ in $(L \cap M) \backslash O$.
By definition of $L$ there exists a $B$-$x$ trail $Q$ in $\casc{B\cup D \cup E}$ that is active given $(L \backslash \{x\}) \cup E$. 
Since $O\subseteq L$, any vertex of $Q$ in $O \cup E$ is a v-structure.
As $M$ is minimal there is a $x$-$D$ trail $R$ that is active given $M \cup E$.
Since $Q$ followed by $R$ is a $B$-$D$ trail in $\casc{B \cup D \cup O \cup E}$, and $Q$ does not intersect $O \cup E$ on a vertex which is not a v-structure, by Lemma~\ref{lem:BDtrailintersectDseparatingSet}, there is a non v-structure of $R$ in $O$. 
Starting from $x$, let $y$ be the last such vertex. 
Let $T$ be the $y$-$D$ subtrail of $R$. 
Note that $R$ can intersect $M \cup E$ only on v-structures, and hence $y \notin M$ and $T$ can intersect $M$ only on v-structures.
As $y \in L = \mb_{M \cup N}(B|E)$, there is a $B$-$y$ trail $S$ in $\casc{B \cup D \cup E}$ that intersects $M$ only on v-structures. 
Hence, $S$ followed by $T$ is a $B$-$D$ trail in $\casc{B \cup D \cup E \cup M}$ that intersects $M \cup E$ only on v-structures, and Lemma~\ref{lem:BDtrailintersectDseparatingSet} gives a contradiction.
\end{proof}

\begin{lem}\label{lem:intermediate_lemma_dseparation}
	Let $M$ and $N$ be two d-separators between $B$ and $D$ given $E$. If $N$ is minimal and $M \perp D | N \cup E$, then
	\begin{equation}\label{eq:dsep_intermediate_lemma}
	 	B \perp N | M \cup E
	 \end{equation} 
\end{lem}
\begin{proof}
Suppose that $B \nperp N | M \cup E$. Let $x$ be a vertex of $N \backslash M$ that is not d-separated from $B$ given $M \cup E$, and $Q$ be a $B$-$x$ trail that is active given $M \cup E$. 
As $N$ is minimal, $N \subseteq \casc{B \cup D \cup E}$ and there is an $x$-$D$ trail $R$ that is active given $N \cup E$.
This trail does not intersect $M$ as this would contradict $M \perp D | N \cup E$.
Hence $Q$ followed by $R$ is a $B$-$D$ trail in $\casc{B \cup D \cup E \cup M}$ that intersects $M \cup E$ only on v-structures, which gives a contradiction.
\end{proof}

\begin{proof}[Proof of Proposition \ref{prop:equivalence_def_directionalMB} in directed graphical models]
	Let $M$ be a minimal d-separator between $B$ and $D$ given $E$.
	
	We start by proving ``not \eqref{eq:nearestDseparator_newversion}'' implies ``not \eqref{eq:nearestDseparator}''.
	Suppose that there exists a minimal d-separator $M'$ such that $B \nperp M' | M \cup E$. Since $M'$ is minimal, Lemma \ref{lem:intermediate_lemma_dseparation} ensures that $D \nperp M | M' \cup E$. There exists a d-separator $M'$ such that $D \nperp M | M' \cup E$.


	We now prove ``not \eqref{eq:nearestDseparator}'' implies ``not \eqref{eq:nearestDseparator_newversion} ''.
	Let $M$ be a minimal d-separator, and $M'$ be a d-separator such that $D \nperp M | M' \cup E$.  
	Let $M'' = M' \cap \casc{B \cup D \cup E}$.
	Let $M_1 = \mb_{\mb_{M \cup M''}(B|E)}(D|E)$. 
	Lemma~\ref{lem:IntersectionWithAscendantsIsStillDseparating} ensures that $B \perp D | M'' \cup E$. Since $M'' \subseteq \casc{B \cup D \cup E \cup M}$ Lemma~\ref{lem:addingVerticesDseparator} ensures that $M \cup M''$ is a d-separator between $B$ and $D$ given $E$. Hence Lemma~\ref{lem:minimalDseparatingSetInM} ensures that $M_1$ is a minimal d-separator between $B$ and $D$ given $E$.
	To prove ``not \eqref{eq:nearestDseparator_newversion}'', we prove $B \nperp M_1 | M \cup E$.	

	Let $x$ be a vertex of $M$ such that $x \nperp D | M' \cup E$. We start by proving $x \perp B | M'' \cup E$. 
	Let $Q$ be a $B$-$x$ trail. We prove that $Q$ is not active given $M'' \cup E$.
	Since, $x \in \casc{B\cup D \cup E}$ and $B \subseteq \casc{B\cup D \cup E}$, if $Q$ intersects $V \backslash \casc{B \cup D \cup E}$, then it contains a v-structure in $V \backslash \casc{B \cup D \cup E}$ which cannot be active given $M'' \cup E$ because $M''\subseteq \casc{B \cup D \cup E}$.
	Suppose now that $Q$ is in $\casc{B\cup D \cup E}$, and let $R$ be an $x$-$D$ trail that is active given $M' \cup E$. As $M'$ d-separates $B$ and $D$ given $E$, Lemma~\ref{lem:BDtrailintersectDseparatingSet} ensures that $Q$ followed by $R$ intersects $M'$ on a non-v structure. This intersection is necessarily in $Q$ and in $M''$. Hence $Q$ is not active given $M'' \cup E$. 
	And we have proved $x \perp B | M'' \cup E$.


	
	We now prove that $x$ does not belong to $M_1$.
	By Lemma~\ref{lem:elementOfMinimalIsDseparatedByBlanket}, it suffices to prove that $x$ does not belong to $\mb_{M \cup M''}(B|E)$.
	Suppose that there is a $B$-$x$ trail active given $(M \cup M'' \backslash \{x\}) \cup E$. 
	Let $P$ be such a trail with a minimal number of v-structure.
	Remark that $P$ is in $\casc{B \cup D \cup E}$.
	Let $b_0$ be the first vertex of the trail starting from $B$.
	Let $s_1,\ldots,s_k$ be the v-structure of $P$ that have no descendants in $M'' \cup E$.
	We prove recursively that $s_i$ has a descendant $b_i$ in $B$.
	Indeed, $s_i$ has either a descendant in $B$ or in $D$.
	By iteration hypothesis, it cannot have a descendant in $D$ as otherwise we would have a $b_{i-1}$-$D$ trail that is active given $M'' \cup E$.
	Hence it has a descendant $b_i$ in $B$, with gives the iteration hypothesis.
	Hence there is a $b_k$-$x$ path that is active given $M'' \cup E$, which gives a contradiction.


	The set $M \backslash M_1$ contains $x$ and is therefore non-empty. Theorem~\ref{theo:characterization_mbdirectional} ensures that $M_1$ satisfies $M_1 \perp D | M \cup E$. Since $M$ is a minimal d-separator between $B$ and $D$ given $E$, Lemma~\ref{lem:intermediate_lemma_dseparation} ensures that $B \perp M | M_1 \cup E$.
	Proposition \ref{prop:mbInCcupIndep} ensures that $\mb_{M \cup M_1}(B|E) = \mb_{M_1}(B|E)$.
	As $M_1$ is a minimal d-separator between $B$ and $D$ given $E$, Corollary \ref{coro:mbset_dsep_min} ensures that $\mb_{M_1}(B|E) = M_1$. We deduce that $\mb_{M \cup M_1}(B|E) = M_1$.
	We therefore cannot have $M_1 \perp B | M \cup E$, as this would imply $M_1 = \mb_{M \cup M_1}(B|E) = \mb_{M}(B|E)= M$,
	which gives ``not \eqref{eq:nearestDseparator_newversion}''.
\end{proof}

\bibliographystyle{plainnat}
\bibliography{stoOptimGraphModel}


\end{document}